\documentclass[11pt]{amsart}
\usepackage[margin=1.25in]{geometry} 
\usepackage[T1]{fontenc} 

\usepackage[utf8]{inputenc}
\usepackage[T1]{fontenc}
\usepackage{amsmath}
\usepackage{amssymb}
\usepackage{amsthm}
\usepackage{hyperref}
\usepackage{mathtools}
\theoremstyle{plain}
\newtheorem{theorem}{Theorem}[section]
\newtheorem{proposition}[theorem]{Proposition}
\newtheorem{lemma}[theorem]{Lemma}
\newtheorem{corollary}[theorem]{Corollary}

\theoremstyle{definition}
\newtheorem{definition}[theorem]{Definition}
\newtheorem{example}[theorem]{Example}

\theoremstyle{remark}
\newtheorem{remark}[theorem]{Remark}

\DeclareMathOperator{\Hom}{Hom}

\title[A Note on the Eilenberg-Mac Lane Isomorphism for Quadratic Forms]{A Note on the Eilenberg-Mac Lane Isomorphism for Quadratic Forms}

\author[C. Galindo]{C\'esar Galindo}
\address{Departamento de Matem\'aticas, Universidad de los Andes, Bogot\'a, Colombia}\email{cn.galindo1116@uniandes.edu.co}

\date{\today}


\begin{document}
\dedicatory{To Zhenghan Wang on his 60th birthday with admiration}
\thanks{The author was partially supported by Grant INV-2025-213-3452 from the School of Science of Universidad de los Andes.}
\begin{abstract}
We give an elementary proof of the Eilenberg-Mac Lane trace isomorphism between the third 2-abelian cohomology group and quadratic forms. Our approach yields explicit constructions and we characterize when quadratic forms can be expressed as traces of bilinear forms for arbitrary coefficient groups.
\end{abstract}

\maketitle

\section{Introduction}

The third abelian cohomology groups $H^3_{k\text{-ab}}(G,M)$ for $k \geq 2$ appear in algebraic topology as the second Postnikov invariant of connected homotopy 2-types, in the classification of pointed braided fusion categories \cite{DGNO10}, and in vertex operator algebra theory, where module categories of lattice VOAs correspond to pointed modular categories whose structure involves quadratic forms \cite{gannon2024orbifolds}. When additional symmetry conditions are imposed, leading to symmetric categorical groups, the relevant cohomology group becomes $H^3_{3\text{-ab}}(G,M)$.

Elements of $H^3_{2\text{-ab}}(G,M)$ are called \emph{2-abelian 3-cocycles} to distinguish them from standard group cohomology. These determine pointed braided fusion categories and play a role in constructing other classes of braided fusion categories \cite{ENO11, Natale18}.

A fundamental result of Eilenberg and Mac Lane \cite{EM53} establishes an isomorphism 
\begin{equation*}
\operatorname{Tr}: H^3_{2\text{-ab}}(G,M) \xrightarrow{\sim} \operatorname{Quad}(G,M)
\end{equation*}
between the third 2-abelian cohomology group and the group of quadratic forms. Recently, some authors have presented elementary arguments for the surjectivity of this map (see \cite[Section 11]{Braunling21}, \cite[Appendix D]{DGNO10}, \cite[Proposition 2.5.1]{Quinn}, and \cite[Theorem 12]{JS86report}).

The goal of this note is to present another proof of the Eilenberg-Mac Lane isomorphism. Our method relies on two observations about how quadratic forms and 2-abelian 3-cocycles behave under direct sums (Lemmas \ref{lem:direct_sum_decomposition} and \ref{lem:direct_sum_construction}), allowing us to reduce both surjectivity and injectivity arguments to cyclic groups. This approach provides explicit constructions for the correspondence and addresses a natural question about the representability of quadratic forms.

We characterize when quadratic forms can be expressed as traces of bilinear forms for arbitrary coefficient groups. Using the exact sequence connecting bilinear forms, quadratic forms, and cohomology, we show that this occurs precisely when certain torsion conditions are satisfied (Theorem \ref{thm:individual_quadratic_criterion}). 

This paper is organized as follows. Section \ref{sec:preliminaries} reviews the definitions and basic properties of 2-abelian 3-cocycles and quadratic forms. Section \ref{sec:EM_theorem} contains our proof of the main theorem for finitely generated abelian groups using the decomposition method. Section \ref{sec: general group} extends our results to arbitrary abelian groups by establishing that $H^3_{2\text{-ab}}(-,M)$ preserves filtered colimits. Section \ref{sec:applications} presents applications to the representability problem for quadratic forms and connections to homotopy theory.

\section{Preliminaries}\label{sec:preliminaries}

\subsection{Abelian cohomology and the Eilenberg-Mac Lane complexes}

In \cite{EM53, EM54}, Eilenberg and Mac Lane defined cohomology groups
\begin{equation*}
H^n_{k\text{-ab}}(G, M) := H^{n+1}( K(G,k), M),
\end{equation*}
where $K(G,k)$ denotes an Eilenberg-Mac Lane space. We call these the \emph{$n$-th $k$-abelian cohomology groups} of $G$ with coefficients in $M$. In Eilenberg and Mac Lane's original notation from \cite{EM54}, these groups correspond to $H^n_{k\text{-ab}}(G, M) = H^{n+1}(G,k;M)$.

For $k = 1$, we recover the ordinary group cohomology
\begin{equation*}
H^n(G,M) = \operatorname{Ext}_{\mathbb{Z}[G]}^n(\mathbb{Z},M),
\end{equation*}
and $G$ need not be abelian.

The case $k = 2$ is of particular interest for the study of quadratic forms, pointed braided fusion categories, and braided categorical groups. 

The interpretation of the associator as a 3-cocycle in classifying 2-groups (monoidal categories where all morphisms and objects are invertible) appears in Sinh's 1975 thesis \cite{Sinh75}, written under Grothendieck's supervision. This built on related ideas in Saavedra-Rivano's work on Tannakian categories \cite{Saavedra72}. For a modern exposition of these historical developments, see \cite{Baez23}. Joyal and Street \cite{JS93} later extended these ideas to braided monoidal categories, and this cohomological perspective became central to the classification of pointed braided fusion categories \cite{DGNO10}.

We will focus exclusively on the third cohomology group $H^3_{2\text{-ab}}(G,M)$ in this case. To distinguish from standard group cohomology, we will refer to elements of $Z^3_{2\text{-ab}}(G,M)$ as \emph{2-abelian 3-cocycles} throughout this paper.

\subsection{Standard 3-cocycles}

We recall the classical description of the third group cohomology. A \emph{3-cocycle} $h: G^3 \to M$ for a group $G$ with coefficients in an abelian group $M$ (viewed as a trivial $G$-module) is a function satisfying the cocycle condition
\begin{equation*}\label{eq:standard-3-cocycle}
h(g_1,g_2,g_3) + h(g_1,g_2g_3,g_4) + h(g_2,g_3,g_4) = h(g_1g_2,g_3,g_4) + h(g_1,g_2,g_3g_4)
\end{equation*}
for all $g_1,g_2,g_3,g_4 \in G$. The group of 3-cocycles is denoted $Z^3(G,M)$. 

The \emph{coboundaries} are 3-cocycles of the form $\delta(c)$ where $c: G^2 \to M$ and
\begin{equation*}\label{eq:standard-coboundary}
(\delta c)(g_1,g_2,g_3) = c(g_2,g_3) - c(g_1g_2,g_3) + c(g_1,g_2g_3) - c(g_1,g_2).
\end{equation*}
The group of coboundaries is denoted $B^3(G,M)$, and the quotient 
\begin{equation*}\label{eq:cohomology-quotient}
H^3(G,M) = Z^3(G,M)/B^3(G,M)
\end{equation*}
is the third cohomology group of $G$ with coefficients in $M$.

\subsection{Abelian 3-cocycles}

When $G$ is abelian, the structure of 3-cocycles in the complex $A(G,2)$ becomes richer. A \emph{2-abelian 3-cocycle} for $G$ with coefficients in $M$ is a pair $(h, c)$, where $h \in Z^3(G,M)$ is a standard 3-cocycle and $c: G^2 \to M$ is a function satisfying the compatibility conditions
\begin{align}
h(y, z, x) + c(x, y + z) + h(x, y, z) &= c(x, z) + h(y, x, z) + c(x, y), \label{eq:c-condition-1} \\
-h(z, x, y) + c(x + y, z) - h(x, y, z) &= c(x, z) - h(x, z, y) + c(y, z). \label{eq:c-condition-2}
\end{align}

The group of 2-abelian 3-cocycles is denoted $Z^3_{2\text{-ab}}(G,M)$. For any function $k: G^2 \to M$, the \emph{2-abelian coboundary} of $k$ is the 2-abelian 3-cocycle $\partial(k) = (\delta(k), \operatorname{Alt}(k))$, where
\begin{equation*}\label{eq:coboundary-c}
\operatorname{Alt}(k)(x, y) = k(x, y) - k(y, x).
\end{equation*}
The group of 2-abelian coboundaries is denoted $B^3_{2\text{-ab}}(G,M)$. The quotient 
\begin{equation*}\label{eq:abelian-cohomology-quotient}
H^3_{2\text{-ab}}(G, M) = Z^3_{2\text{-ab}}(G,M)/B^3_{2\text{-ab}}(G,M)
\end{equation*}
is the \emph{third 2-abelian cohomology group} of $G$ with coefficients in $M$.

\begin{example}\label{ex:abelian-3-cocycle-examples}
We describe $H^3_{2\text{-ab}}(G, M)$ for basic groups following \cite{EM54}.

\begin{enumerate}
    \item A 2-abelian 3-cocycle of the form $(0,c)$ corresponds precisely to a bilinear map $c: G \times G \to M$. Two such cocycles $(0,c)$ and $(0,c')$ are cohomologous if and only if $c - c' = \operatorname{Alt}(k)$ for some standard 2-cocycle $k \in Z^2(G,M)$.
    
    \item For $G = \mathbb{Z}$, we use the standard free resolution
    \begin{equation*}\label{eq:cyclic-resolution}
    0 \longrightarrow \mathbb{Z}[G] \xrightarrow{\sigma-1} \mathbb{Z}[G] \xrightarrow{\epsilon} \mathbb{Z} \longrightarrow 0,
    \end{equation*}
    where $\epsilon(\sigma) = 1$ and $\sigma$ is the generator of $G$. Since $H^n(\mathbb{Z},M) = 0$ for $n \geq 2$, every 2-abelian 3-cocycle is cohomologous to exactly one of the form $(0, c_m)$ with $c_m(x,y) = xym$ for some $m \in M$. Thus $H^3_{2\text{-ab}}(\mathbb{Z},M) \cong M.$
    
    \item For $G = \mathbb{Z}/n\mathbb{Z}$, let $m \in M[\gcd(2n,n^2)]$, where $M[d] = \{m \in M : dm = 0\}$ denotes the $d$-torsion subgroup. The pair $(h_m, c_m)$ given by 
\begin{equation}\label{eq:cocycle-cyclic}
\begin{split}
h_{nm}(x,y,z) &= 
\begin{cases} 
    0 & \text{if } y+z < n, \\
    xnm & \text{if } y+z \geq n,
\end{cases}\\
c_m(x,y) &= xym
\end{split}
\end{equation}
defines a 2-abelian 3-cocycle. As we will establish in Lemma \ref{lem:trace_iso_cyclic_case}, these 2-abelian 3-cocycles provide a complete set of representatives for the cohomology classes in $H^3_{2\text{-ab}}(\mathbb{Z}/n\mathbb{Z}, M)$.
\end{enumerate}
\end{example}

\subsection{Quadratic forms on abelian groups}

Let $G$ and $M$ be abelian groups. A function $q: G \to M$ is called \emph{quadratic} if it satisfies the following conditions:
\begin{enumerate}
\item $q(-a) = q(a)$ for all $a \in G$.
\item The associated map $b_q: G \times G \to M$ defined by 
\begin{equation*}\label{eq:bilinear-form}
b_q(a_1, a_2) = q(a_1 + a_2) - q(a_1) - q(a_2)
\end{equation*}
is a bilinear form.
\end{enumerate}

We denote by $\operatorname{Quad}(G, M)$ the abelian group of all quadratic forms from $G$ to $M$. It follows by induction that $q(nx) = n^2q(x)$ for all $n \in \mathbb{Z}$ and $x \in G$.

\begin{example}\label{ex:quadratic-forms-examples}
We describe $\operatorname{Quad}(G, M)$ for basic groups.
\begin{enumerate}
\item For $G = \mathbb{Z}$, any quadratic form satisfies $q(x) = x^2q(1)$. The evaluation map $q \mapsto q(1)$ gives an isomorphism
\begin{equation*}\label{eq:quad-Z}
\operatorname{Quad}(\mathbb{Z}, M) \cong M.
\end{equation*}

\item For $G = \mathbb{Z}/n\mathbb{Z}$, the condition $q(nx) = n^2q(x)$ implies $n^2 q(1) = q(n \cdot 1) = 0$. From the bilinearity of $b_q$, we also obtain $2nq(1) = b_q(n, 1) = 0$. Both conditions are equivalent to $\gcd(n^2, 2n)q(1) = 0$, where 
\begin{equation*}\label{eq:gcd-computation}
\gcd(n^2, 2n) = \begin{cases}
n & \text{if } n \text{ is odd}, \\
2n & \text{if } n \text{ is even}.
\end{cases}
\end{equation*}
The evaluation map $q \mapsto q(1)$ induces an isomorphism
\begin{equation*}\label{eq:quad-Zn}
\operatorname{Quad}(\mathbb{Z}/n\mathbb{Z}, M) \cong M[\gcd(2n,n^2)] = \begin{cases}
M[n] & \text{if } n \text{ is odd}, \\
M[2n] & \text{if } n \text{ is even},
\end{cases}
\end{equation*}
where $M[s] = \{m \in M : sm = 0\}$ denotes the $s$-torsion subgroup.
\end{enumerate}
\end{example}

The connection between 2-abelian 3-cocycles and quadratic forms is established through the trace map. We include the following result for completeness.

\begin{lemma}\label{lem:cocycle-to-quadratic}
If $(h,c)$ is a 2-abelian 3-cocycle, then $q(x) := c(x,x)$ defines a quadratic form on $G$.
\end{lemma}
\begin{proof}
Let $S(x,y) := c(x,y) + c(y,x)$. Using condition \eqref{eq:c-condition-1} with the substitution $x \mapsto y+z$, we obtain 
\begin{equation*}
c(y+z,y+z) - c(y+z,y) - c(y+z,z) = -h(y+z,y,z) + h(y,y+z,z) - h(y,z,y+z).
\end{equation*}
Applying the standard 3-cocycle condition for $h$ to the right-hand side gives
\begin{equation*}
-h(y+z,y,z) + h(y,y+z,z) - h(y,z,y+z) = -h(y,z,y) - h(z,y,z).
\end{equation*}
Now, applying condition \eqref{eq:c-condition-2} twice yields
\begin{align*}
c(y+z,y) - c(y,y) - c(z,y) &= -h(y,z,y), \\
c(y+z,z) - c(y,z) - c(z,z) &= -h(z,y,z).
\end{align*}
Adding these two equalities to our original equation, the $-h(\cdot)$ terms cancel, and we obtain
\begin{equation}\label{eq: S matrix}
b_q(y,z) = c(y+z,y+z) - c(y,y) - c(z,z) = c(y,z) + c(z,y) = S(y,z).
\end{equation}
Using the cocycle conditions, one verifies that $S$ is bilinear. 

To check $q(-x) = q(x)$, we compute 
\begin{equation*}
S(x,-x) = c(x,-x) + c(-x,x) = -c(x,x) - c(x,x) = -2q(x).
\end{equation*}
Since $b_q(x,-x) = q(0) - q(x) - q(-x) = S(x,-x) = -2q(x)$ and $q(0) = 0$, we conclude $q(-x) = q(x)$.
\end{proof}

\begin{definition}\label{def:trace-map}
The \emph{trace map} $\operatorname{Tr}: H^3_{2\text{-ab}}(G, M) \to \operatorname{Quad}(G, M)$ is defined by
\begin{equation*}
\operatorname{Tr}[(h, c)](x) = c(x, x)
\end{equation*}
for any representative $(h, c)$ of a cohomology class.
\end{definition}

\section{The Eilenberg-Mac Lane Isomorphism for finitely generated groups}\label{sec:EM_theorem}

The goal of this section is to establish the Eilenberg-Mac Lane isomorphism for finitely generated abelian groups using a constructive approach based on direct sum decompositions. Our method reduces the general case to cyclic groups through two lemmas that characterize how quadratic forms and 2-abelian 3-cocycles behave under direct sums.

We denote by $\operatorname{Bil}(G_1 \times G_2, M)$ the abelian group of bilinear forms $f: G_1 \times G_2 \to M$, that is, functions satisfying $f(g_1 + g'_1, g_2) = f(g_1, g_2) + f(g'_1, g_2)$ and $f(g_1, g_2 + g'_2) = f(g_1, g_2) + f(g_1, g'_2)$ for all $g_1, g'_1 \in G_1$ and $g_2, g'_2 \in G_2$.

We begin with the following decomposition result for quadratic forms.

\begin{lemma}\label{lem:direct_sum_decomposition}
Let $G_1$, $G_2$, and $M$ be abelian groups. The map
\begin{equation*}
\Phi : \operatorname{Bil}(G_1 \times G_2, M) \oplus \operatorname{Quad}(G_1, M) \oplus \operatorname{Quad}(G_2, M) \to \operatorname{Quad}(G_1 \oplus G_2, M)
\end{equation*}
defined by $\Phi(f, q_{G_1}, q_{G_2})(g_1, g_2) := f(g_1, g_2) + q_{G_1}(g_1) + q_{G_2}(g_2)$ is an isomorphism.
\end{lemma}

\begin{proof}
We construct the inverse map $\Psi: \operatorname{Quad}(G_1 \oplus G_2, M) \to \operatorname{Bil}(G_1 \times G_2, M) \oplus \operatorname{Quad}(G_1, M) \oplus \operatorname{Quad}(G_2, M)$ as follows. For any quadratic form $q: G_1 \oplus G_2 \to M$, define
\begin{align*}
q_{G_1}(g_1) &:= q(g_1,0), \\
q_{G_2}(g_2) &:= q(0,g_2), \\
f_q(g_1,g_2) &:= q(g_1,g_2) - q(g_1,0) - q(0,g_2).
\end{align*}
One verifies directly that $q_{G_1}$ and $q_{G_2}$ are quadratic forms and $f_q$ is bilinear. The maps $\Phi$ and $\Psi$ are clearly inverse to each other.
\end{proof}

The following lemma shows how to construct 2-abelian 3-cocycles on direct sums.

\begin{lemma}\label{lem:direct_sum_construction}
Let $G_1$, $G_2$, and $M$ be abelian groups. Let $(h_i, c_i) \in Z^3_{2\text{-ab}}(G_i, M)$ for $i = 1, 2$, and let $f \in \operatorname{Bil}(G_1 \times G_2, M)$. Define $(h, c)$ on $G := G_1 \oplus G_2$ by
\begin{equation}\label{eq:direct_sum_cocycle}
\begin{split}
h((x_1,y_1), (x_2,y_2), (x_3,y_3)) &= h_1(x_1, x_2, x_3) + h_2(y_1, y_2, y_3), \\
c((x_1,y_1), (x_2,y_2)) &= c_1(x_1, x_2) + c_2(y_1, y_2) + f(x_1, y_2).
\end{split}
\end{equation}
Then $(h, c) \in Z^3_{2\text{-ab}}(G, M)$ and $\operatorname{Tr}([(h, c)]) = \Phi(f,q_1,q_2)$, where $q_i = \operatorname{Tr}([(h_i, c_i)])$.
\end{lemma}

\begin{proof}
We verify that $(h,c)$ satisfies the 2-abelian 3-cocycle conditions. Any bilinear map $f: G_1 \times G_2 \to M$ defines a 2-abelian 3-cocycle $(0, c_f)$ on $G = G_1 \oplus G_2$ with $c_f((x_1,y_1), (x_2,y_2)) = f(x_1, y_2)$. Let $\pi_i: G \to G_i$ be the canonical projections. Then 
\begin{equation*}
(h, c) = \pi_1^*(h_1, c_1) + \pi_2^*(h_2, c_2) + (0, c_f)
\end{equation*}
is a sum of 2-abelian 3-cocycles, hence $(h, c) \in Z^3_{2\text{-ab}}(G, M)$. The trace formula follows from 
\begin{equation*}
\operatorname{Tr}([(h, c)])((x, y)) = c((x,y), (x,y)) = c_1(x,x) + c_2(y,y) + f(x, y) = q_1(x) + q_2(y) + f(x, y).
\end{equation*}
\end{proof}

\begin{remark}
Lemma \ref{lem:direct_sum_decomposition} shows that quadratic forms on direct sums decompose canonically into quadratic forms on the summands plus a bilinear cross-term. Combined with Lemma \ref{lem:direct_sum_construction}, this decomposition allows us to reduce the study of the trace map to cyclic groups.
\end{remark}

We now establish the base case of our inductive proof by showing that the trace map is an isomorphism for cyclic groups. The following result follows closely the approach in \cite{JS93} and \cite{EM54}. We include a proof for completeness.

\begin{lemma}\label{lem:trace_iso_cyclic_case}
Let $C$ be a cyclic group. Then the trace map $\operatorname{Tr}: H^3_{2\text{-ab}}(C, M) \to \operatorname{Quad}(C, M)$ is an isomorphism.   
\end{lemma}

\begin{proof}
If $C \cong \mathbb{Z}$, we showed in Example \ref{ex:abelian-3-cocycle-examples}(2) that every cohomology class has a unique representative $(0, c_m)$ with $c_m(x,y) = mxy$. Since $\operatorname{Tr}([(0, c_m)])(x) = x^2m$, the map $\operatorname{Tr}$ is an isomorphism.

If $C \cong \mathbb{Z}/n\mathbb{Z}$, we have $\operatorname{Quad}(\mathbb{Z}/n\mathbb{Z}, M) \cong M[\gcd(2n,n^2)]$ by Example \ref{ex:quadratic-forms-examples}(2). The abelian 3-cocycles $(h_{nm}, c_m)$ from \eqref{eq:cocycle-cyclic} satisfy $\operatorname{Tr}([(h_{nm}, c_m)])(1) = m$, establishing surjectivity.

For injectivity, we use that $H^3(\mathbb{Z}/n\mathbb{Z},M) \cong M[n]$ with representative 3-cocycle
\begin{align*}
    h_{m}(x,y,z) &= 
\begin{cases} 
    0 & \text{if } y+z < n \\
    xm & \text{if } y+z \geq n 
\end{cases}
\end{align*}
for $m \in M[n]$. Suppose $\operatorname{Tr}([(h, c)]) = 0$. We may assume $(h, c)$ has the form $(h_m, c)$ for some $m \in M$. Then $c(1, x+y) = c(1,x) + c(1,y) + h(1,x,y)$ when $x+y < n$. Since $c(1,x) = xc(1,1) = 0$ for $x < n$, we get $0 = c(1,n) = h_m(1,1,n-1) = m$. Hence $h_m = 0$ and $c$ is bilinear with $c(1,1) = 0$, which implies $c = 0$.
\end{proof}
We can now prove the main result of this section.

\begin{theorem}[\cite{EM54}]\label{thm:main_finitely_generated}
Let $G$ be a finitely generated abelian group and $M$ an abelian group. The trace map 
\begin{equation*}
\operatorname{Tr}: H^3_{2\text{-ab}}(G, M) \xrightarrow{\sim} \operatorname{Quad}(G, M)
\end{equation*}
is an isomorphism.
\end{theorem}

\begin{proof}
We proceed by induction on the minimal number of generators of $G$. The base case of cyclic groups is established in Lemma \ref{lem:trace_iso_cyclic_case}. 

For the inductive step, assume $G = G_1 \oplus G_2$ where both $G_1$ and $G_2$ require fewer generators than $G$. By the induction hypothesis, the trace maps $\operatorname{Tr}: H^3_{2\text{-ab}}(G_i, M) \to \operatorname{Quad}(G_i, M)$ are isomorphisms for $i = 1, 2$.

To prove surjectivity, let $q \in \operatorname{Quad}(G, M)$. By Lemma \ref{lem:direct_sum_decomposition}, $q$ has the unique decomposition $q = \Phi(f, q_1, q_2)$ for some $q_i \in \operatorname{Quad}(G_i, M)$ and $f \in \operatorname{Bil}(G_1 \times G_2, M)$. By the induction hypothesis, each $q_i$ determines a unique cohomology class $[(h_i, c_i)] \in H^3_{2\text{-ab}}(G_i, M)$ with $\operatorname{Tr}([(h_i, c_i)]) = q_i$. Using Lemma \ref{lem:direct_sum_construction}, we construct $(h, c) \in Z^3_{2\text{-ab}}(G, M)$ such that $\operatorname{Tr}([(h, c)]) = \Phi(f,q_1,q_2) = q$. This establishes surjectivity.

For injectivity, suppose $\operatorname{Tr}([(h, c)]) = 0$ for some $[(h,c)] \in H^3_{2\text{-ab}}(G, M)$. By the decomposition in Lemma \ref{lem:direct_sum_decomposition}, this implies $q_i = 0$ and $f = 0$ in the corresponding decomposition. By the induction hypothesis, we have $[(h_i, c_i)] = 0$ in $H^3_{2\text{-ab}}(G_i, M)$. By Lemma \ref{lem:direct_sum_construction}, this forces $[(h, c)] = 0$ in $H^3_{2\text{-ab}}(G, M)$.
\end{proof}

\section{The Eilenberg-Mac Lane isomorphism for general groups}\label{sec: general group}

In \cite[Theorem 12]{JS86report} (later published as \cite{JS93}), the authors outline an approach to prove the Eilenberg-Mac Lane isomorphism for arbitrary abelian groups using colimits of finitely generated subgroups, under the assumption that $H^3_{2\text{-ab}}(-,M)$ preserves filtered colimits. A similar strategy appears in the original work \cite{EM54}, where Eilenberg and Mac Lane work with homology and then use the universal coefficient theorem, since homology naturally respects colimits, rather than working directly with cohomology. In the published version of \cite{JS93}, the authors ultimately cite the result from \cite{EM54} directly.

The goal of this section is to establish that $H^3_{2\text{-ab}}(-,M)$ does preserve filtered colimits, thereby completing the approach suggested in \cite{JS86report}. We first demonstrate through a concrete example that the analogous property fails for $H^2_{2\text{-ab}}(-,M) = \text{Ext}^1(-,M)$, showing that this continuity property is non-trivial and requires proof.

\begin{example}\label{ex:H2_not_continuous}
Consider $\mathbb{Q}$, the group of rational numbers. For each $n \in \mathbb{N}$, define
\begin{equation*}
A_n = \frac{1}{n!}\mathbb{Z} = \left\{\frac{k}{n!} : k \in \mathbb{Z}\right\}.
\end{equation*}
Note that every finitely generated subgroup of $\mathbb{Q}$ is contained in some $A_n$ for $n$ sufficiently large, each $A_n \cong \mathbb{Z}$, and $\mathbb{Q} = \bigcup_{n=1}^{\infty} A_n$.

Since each $A_n$ is an infinite cyclic group, we have $H^i(A_n,\mathbb{Z}) = 0$ for all $i \geq 2 $. In particular, $\text{Ext}^1(A_n,\mathbb{Z}) = H^2_{2\text{-ab}}(A_n,\mathbb{Z}) = 0$. Therefore
\begin{equation*}
\varprojlim H^2_{2\text{-ab}}(A_n, \mathbb{Z}) = 0.
\end{equation*}

On the other hand, using the exact sequence $0 \to \mathbb{Z} \to \mathbb{Q} \to \mathbb{Q}/\mathbb{Z} \to 0$, and applying the contravariant functor $\Hom(\mathbb{Q},-)$, we obtain the long exact sequence
\begin{equation*}
0 \to \Hom(\mathbb{Q},\mathbb{Z}) \to \Hom(\mathbb{Q},\mathbb{Q}) \to \Hom(\mathbb{Q},\mathbb{Q}/\mathbb{Z}) \to \text{Ext}^1(\mathbb{Q},\mathbb{Z}) \to \text{Ext}^1(\mathbb{Q},\mathbb{Q}) = 0.
\end{equation*}
Since $\Hom(\mathbb{Q},\mathbb{Z}) = 0$ and $\Hom(\mathbb{Q},\mathbb{Q}) = \mathbb{Q}$, we get 
\begin{equation*}
\text{Ext}^1(\mathbb{Q},\mathbb{Z}) = \Hom(\mathbb{Q},\mathbb{Q}/\mathbb{Z})/\mathbb{Q},
\end{equation*}
which is an uncountable group.

This shows that $\varprojlim \text{Ext}^1(A_n, \mathbb{Z}) = 0$ but $\text{Ext}^1(\mathbb{Q},\mathbb{Z}) \neq 0$, proving that $H^2_{2\text{-ab}}(-,\mathbb{Z})$ does not preserve filtered colimits.
\end{example}
\begin{remark}
In what follows, we use the notation $(H_{2\text{-ab}})_n(G,\mathbb{Z})$ to denote the $n+1$-th homology group of the Eilenberg-Mac Lane space $K(G,2)$ with integer coefficients. Recall that $H^n_{2\text{-ab}}(G,M) = H^{n+1}(K(G,2), M)$ by definition. The universal coefficient theorem relates the homology of $K(G,2)$ to the cohomology groups $H^n_{2\text{-ab}}(G,M)$ for arbitrary coefficient groups $M$, which is the key tool in the proof below.

We emphasize that the following proof works for filtered colimits of arbitrary cardinality, not just countable index families. The argument does not rely on any Mittag-Leffler type conditions. This generality is essential for Theorem \ref{thm:main_general}, where the index set parametrizing finitely generated subgroups of an arbitrary abelian group $G$ may have any cardinality.

\end{remark}
\begin{lemma}\label{lem:H3_ab_continuous}
Let $M$ be an abelian group. The contravariant functor $G \mapsto H^3_{2\text{-ab}}(G,M)$ preserves filtered colimits, that is, 
\begin{equation*}
H^3_{2\text{-ab}}\left(\varinjlim G_i, M\right) \cong \varprojlim H^3_{2\text{-ab}}(G_i, M).
\end{equation*}
\end{lemma}

\begin{proof}
By the universal coefficient theorem, we have the exact sequence 
\begin{equation}\label{eq:universal_coefficient}
0 \to \text{Ext}^1(M, (H_{2\text{-ab}})_{n-1}(G,\mathbb{Z})) \to H^n_{2\text{-ab}}(G,M) \to \Hom((H_{2\text{-ab}})_{n}(G,\mathbb{Z}), M) \to 0    
\end{equation}
for arbitrary abelian groups $G$ and $M$.

In the case $n = 2$, using $(H_{2\text{-ab}})_1(G,\mathbb{Z}) = H_1(G,\mathbb{Z})=G$, the exact sequence \eqref{eq:universal_coefficient} becomes 
\begin{equation*}
H^2_{2\text{-ab}}(G,M) = \text{Ext}^1(M,G),
\end{equation*}for every abelian groups $M$. This implies $(H_{2\text{-ab}})_2(G,\mathbb{Z}) = 0$ for any abelian group $G$. 

Now, since $(H_{2\text{-ab}})_2(G,\mathbb{Z}) = 0$ applying \eqref{eq:universal_coefficient} with $n = 3$, we obtain
\begin{equation}\label{eq:homology_3}
H^3_{2\text{-ab}}(G,M) \cong \Hom((H_{2\text{-ab}})_3(G,\mathbb{Z}), M),
\end{equation}
naturally in $G$. 

The functor $H^3_{2\text{-ab}}(-,M)$ is therefore the composition of the functor $(H_{2\text{-ab}})_3(-,\mathbb{Z})$, which commutes with colimits since homology preserves colimits, and the contravariant functor $\Hom(-,M)$, which converts colimits to limits. We thus have the natural isomorphisms:
\begin{align*}
H^3_{2\text{-ab}}\left(\varinjlim G_i, M\right) &\cong \Hom\left((H_{2\text{-ab}})_3\left(\varinjlim G_i, \mathbb{Z}\right), M\right) \\
&\cong \Hom\left(\varinjlim (H_{2\text{-ab}})_3(G_i, \mathbb{Z}), M\right) \\
&\cong \varprojlim \Hom((H_{2\text{-ab}})_3(G_i, \mathbb{Z}), M) \\
&\cong \varprojlim H^3_{2\text{-ab}}(G_i, M).
\end{align*}
\end{proof}

\begin{theorem}[\cite{EM54}]\label{thm:main_general}
Let $G$ be an abelian group and $M$ an abelian group. The trace map 
\begin{equation*}
\operatorname{Tr}: H^3_{2\text{-ab}}(G, M) \to \operatorname{Quad}(G, M) 
\end{equation*}
is an isomorphism.
\end{theorem}

\begin{proof}
We follow the approach outlined in \cite[Theorem 12]{JS86report}, which is essentially the same as in \cite{EM54}. The key observation is that every abelian group is a filtered colimit of its finitely generated subgroups.

Let $G$ be an arbitrary abelian group and let $\{G_\alpha\}_{\alpha \in I}$ denote the directed system of finitely generated subgroups of $G$ ordered by inclusion. Then $G = \varinjlim_{\alpha \in I} G_\alpha$.

By Theorem \ref{thm:main_finitely_generated}, the trace map is an isomorphism for each finitely generated group $G_\alpha$. Therefore, the induced map on inverse limits
\begin{equation*}
\varprojlim_{\alpha \in I} H^3_{2\text{-ab}}(G_\alpha, M) \to \varprojlim_{\alpha \in I} \operatorname{Quad}(G_\alpha, M)
\end{equation*}
is an isomorphism.

The functor $\operatorname{Quad}(-, M)$ converts colimits to limits since quadratic forms are determined by their values on elements. More precisely, for a directed system $\{G_\alpha\}_{\alpha \in I}$, a quadratic form on $\varinjlim G_\alpha$ is equivalent to a compatible family of quadratic forms on the $G_\alpha$'s. Thus

\begin{equation*}
\operatorname{Quad}\left(\varinjlim_{\alpha \in I} G_\alpha, M\right) \cong \varprojlim_{\alpha \in I} \operatorname{Quad}(G_\alpha, M).
\end{equation*}

By Lemma \ref{lem:H3_ab_continuous}, we also have 
\begin{equation*}
H^3_{2\text{-ab}}\left(\varinjlim_{\alpha \in I} G_\alpha, M\right) \cong \varprojlim_{\alpha \in I} H^3_{2\text{-ab}}(G_\alpha, M).
\end{equation*}

Combining these isomorphisms, we conclude that the trace map $\operatorname{Tr}: H^3_{2\text{-ab}}(G, M) \to \operatorname{Quad}(G, M)$ is an isomorphism for arbitrary abelian groups $G$.
\end{proof}

We conclude this section with an observation that connects our approach to the original Eilenberg-Mac Lane construction. The functor $G \mapsto \operatorname{Quad}(G,M)$ can be represented by a universal quadratic form. Specifically, there exists a universal quadratic group $\Gamma(G)$ together with a universal quadratic map $\gamma: G \to \Gamma(G)$ such that for any quadratic map $\psi: G \to M$, there exists a unique group homomorphism $\Psi: \Gamma(G) \to M$ satisfying $\Psi \circ \gamma = \psi$. This universal property gives $\operatorname{Quad}(G,M) \cong \Hom(\Gamma(G), M)$ for all abelian groups $M$.

The functor $\Gamma$ is precisely the universal quadratic functor that appears in Whitehead's exact sequence \cite{Whitehead1950}.

\begin{corollary}
For any abelian group $G$, there is a natural isomorphism 
\begin{equation*}
\Gamma(G) \cong (H_{2\text{-ab}})_3(G,\mathbb{Z}).
\end{equation*}
\end{corollary}

\begin{proof}
By the universal property of $\Gamma(G)$ and equation \eqref{eq:homology_3}, we have the natural isomorphisms 
\begin{equation*}
\operatorname{Quad}(G,M) \cong \Hom(\Gamma(G), M) \cong \Hom((H_{2\text{-ab}})_3(G,\mathbb{Z}), M)
\end{equation*}
for all abelian groups $M$. By Yoneda's lemma, this implies $\Gamma(G) \cong (H_{2\text{-ab}})_3(G,\mathbb{Z})$.
\end{proof}

\begin{remark}
Eilenberg and Mac Lane's original proof follows precisely this direction: they show that the third homology of the complex $A(G,2)$ is exactly $\Gamma(G)$ and construct the universal map $\gamma: G \to \Gamma(G)$. Moreover, $\Gamma(-)$ is a quadratic functor that preserves colimits, which allows them to reduce their calculations to finitely generated groups and ultimately to cyclic groups.

Our proof follows essentially the opposite direction to that presented by Eilenberg and Mac Lane. However, given our interest in cohomology for applications to fusion categories and homotopy 2-types, and in having specific representatives of cocycles, this approach seems equally natural and more elementary.
\end{remark}
\section{Applications}\label{sec:applications}

\subsection{When quadratic forms are traces of bilinear forms}

For any abelian group $G$ and $M$, we have the natural inclusion $\operatorname{Hom}(G \otimes G, M) \subset Z^3_{2\text{-ab}}(G,M)$ given by $c \mapsto (0,c)$. The trace map $\operatorname{Tr}:Z^3_{2\text{-ab}}(G,M)\to \operatorname{Quad}(G,M)$ restricts to a map $\operatorname{Hom}(G \otimes G, M)\to \operatorname{Quad}(G,M)$, which we also denote by $\operatorname{Tr}$. Via the Eilenberg-Mac Lane isomorphism, this gives rise to the exact sequence
\begin{equation*}\label{eq:exact_sequence_bilinear}
0 \to \operatorname{Hom}(\Lambda^2 G, M) \to \operatorname{Hom}(G \otimes G, M) \xrightarrow{\operatorname{Tr}} \operatorname{Quad}(G,M) \xrightarrow{\cong} H^3_{2\text{-ab}}(G,M) \xrightarrow{s} H^3(G,M),
\end{equation*}
where the isomorphism is the inverse of the trace map on cohomology, and $s: H^3_{2\text{-ab}}(G,M) \to H^3(G,M)$ is the forgetful map given by $s([(h,c)]) = [h]$.

A natural question arising from this sequence concerns when individual quadratic forms can be expressed as traces of bilinear forms.

\begin{lemma}\label{lem:equivalence_bilinear_representability}
For abelian groups $G$ and $M$, the following conditions are equivalent:
\begin{enumerate}
\item  The map $s: H^3_{2\text{-ab}}(G,M) \to H^3(G,M)$ is the zero map.
\item The sequence
\begin{equation*}\label{eq:short_exact_bilinear}
0 \to \operatorname{Hom}(\Lambda^2 G, M) \to \operatorname{Hom}(G \otimes G, M) \xrightarrow{\operatorname{Tr}} \operatorname{Quad}(G,M) \to 0
\end{equation*}
is exact.
\item For every quadratic form $q \in \operatorname{Quad}(G,M)$, there exists a bilinear form $c: G \times G \to M$ such that $q(x) = c(x,x)$ for all $x \in G$.
\end{enumerate}
\end{lemma}

\begin{proof}
The equivalence of (1) and (2) follows immediately from the exact sequence \eqref{eq:exact_sequence_bilinear}. The equivalence of (2) and (3) is a direct translation of the exactness condition in terms of quadratic forms and bilinear forms.
\end{proof}

The following lemma shows how this property behaves under direct sums.

\begin{lemma}\label{lem:suspension_image_decomposition}
Under the isomorphism $\Phi : \operatorname{Bil}(G_1 \times G_2, M) \oplus \operatorname{Quad}(G_1, M) \oplus \operatorname{Quad}(G_2, M) \to \operatorname{Quad}(G_1 \oplus G_2, M)$ of Lemma \ref{lem:direct_sum_decomposition}, we have $\operatorname{Bil}(G_1 \times G_2, M) \subset \ker(s)$. In particular, for any decomposition $G = \bigoplus_i G_i$, we have $\operatorname{Im}(s_G) = \bigoplus_i \operatorname{Im}(s_{G_i})$, where $s_G: H^3_{2\text{-ab}}(G,M) \to H^3(G,M)$ denotes the forgetful map.
\end{lemma}

\begin{proof}
By Lemma \ref{lem:direct_sum_construction}, the 2-abelian 3-cocycle associated to a bilinear form in $\operatorname{Bil}(G_1 \times G_2, M)$ has the form $(0, c_f)$ with trivial 3-cocycle component. Hence it lies in the kernel of $s$.
\end{proof}

We now provide a criterion for when individual quadratic forms can be represented as traces of bilinear forms.

\begin{theorem}\label{thm:individual_quadratic_criterion}
Let $q: G \to M$ be a quadratic form and denote by $|x| \in \mathbb{Z}^{\geq 1} \cup \{\infty\}$ the order of $x \in G$. Then $q$ can be represented as the trace of a bilinear form if and only if $|x|q(x) = 0$ for every element $x \in T_2(G)=\{x\in G| \exists r\in \mathbb{Z}^{\geq 1} \text{ such that } 2^rx=0\}$, the 2-primary component of $G$.
\end{theorem}

\begin{proof}
By Lemma \ref{lem:H3_ab_continuous}, we may reduce to the case of finitely generated groups. For a finitely generated group $G$, we can write $G = \bigoplus_i G_i$ where each $G_i$ is cyclic. By Lemma \ref{lem:suspension_image_decomposition}, the problem reduces to cyclic groups.

\textbf{Case 1: Infinite cyclic groups.} For $G = \mathbb{Z}$, we have $T_2(G) = \{0\}$, so the condition is vacuous. By Example \ref{ex:abelian-3-cocycle-examples}(2), every 2-abelian 3-cocycle has the form $(0, c_m)$ with $c_m(x,y) = xym$. Hence every quadratic form can be represented as the trace of a bilinear form.

\textbf{Case 2: Finite cyclic groups of odd order.} For $G = \mathbb{Z}/n\mathbb{Z}$ with $n$ odd, we have $T_2(G) = \{0\}$. The condition is again vacuous, and by Example \ref{ex:abelian-3-cocycle-examples}(3), every 2-abelian 3-cocycle can be represented by a bilinear form.

\textbf{Case 3: Finite cyclic groups of $2^k$ order.} For $G = \mathbb{Z}/2^k\mathbb{Z}$ with $k \geq 1$, let $q \in \operatorname{Quad}(G,M)$ be a quadratic form. By Example \ref{ex:quadratic-forms-examples}(2), we have $q(x) = x^2 q(1)$ for some $q(1) \in M[2^{k+1}]$.

If $q$ is the trace of a bilinear form $c$, then $q(1) = c(1,1)$ is the image of a group homomorphism $G \to M$, and hence $2^k q(1) = 0$. The same condition then holds for all other elements of $G$.

Conversely, if $2^k q(1) = 0$, then defining $c(x,y) = xy q(1)$ gives a bilinear form with $\operatorname{Tr}(c) = q$.
\end{proof}

Let us denote by $\operatorname{Quad}_0(G,M) = \{q \in \operatorname{Quad}(G,M) : |x|q(x) = 0 \text{ for all } x \in T_2(G)\}$ the subgroup of quadratic forms satisfying the condition of Theorem \ref{thm:individual_quadratic_criterion} on the 2-primary component, or equivalently, exactly those that are in the kernel of the forgetful map.

\begin{proposition}\label{prop:image_suspension_map}
The image of the forgetful map $s: H^3_{2\text{-ab}}(G,\mathbb{Q}/\mathbb{Z}) \to H^3(G,\mathbb{Q}/\mathbb{Z})$ is naturally isomorphic to $\operatorname{Hom}(T_2(G)/2T_2(G), \mathbb{Q}/\mathbb{Z})$.
\end{proposition}

\begin{proof}
We define a homomorphism $\theta: \operatorname{Quad}(T_2(G), \mathbb{Q}/\mathbb{Z}) \to \operatorname{Hom}(T_2(G)/2T_2(G), \mathbb{Q}/\mathbb{Z})$ by 
\begin{equation*}
\theta(q)(x + 2T_2(G)) := |x|q(x)
\end{equation*}
for $x \in T_2(G)$.

First, we verify that $\theta(q)$ is well-defined. If $x$ has order $2n$, then $2x$ has order $n$, and thus 
\begin{equation*}
|2x|q(2x) = n \cdot 4q(x) = 4nq(x) = 2|x|q(x) = 0.
\end{equation*}
Therefore, the function $\theta(q): T_2(G)/2T_2(G) \to \mathbb{Q}/\mathbb{Z}$ is well-defined.

Next, we show that $\theta(q)$ is a group homomorphism. Note that elements in $\operatorname{Quad}_0(T_2(G), \mathbb{Q}/\mathbb{Z})$ map to zero under $\theta$. In particular, by the decomposition of Lemma \ref{lem:direct_sum_construction}, quadratic forms arising from bilinear forms in direct sum decompositions map to zero under $\theta$. Therefore, it suffices to consider the cyclic case, where for any $x \in T_2(G)$, we have $2\theta(q)(x + 2T_2(G)) = 2|x|q(x) = 0$.

By Theorem \ref{thm:individual_quadratic_criterion}, the kernel of $s$ consists precisely of those quadratic forms in $\operatorname{Quad}_0(G, \mathbb{Q}/\mathbb{Z})$. The image of $s$ corresponds to the quotient $\operatorname{Quad}(T_2(G), \mathbb{Q}/\mathbb{Z})/\operatorname{Quad}_0(T_2(G), \mathbb{Q}/\mathbb{Z})$, and the map $\theta$ induces the desired isomorphism between this quotient and $\operatorname{Hom}(T_2(G)/2T_2(G), \mathbb{Q}/\mathbb{Z})$.
\end{proof}

\begin{corollary}\label{cor:global_characterization}
Let $G$ be an abelian group. The following conditions are equivalent:
\begin{enumerate}
\item For every abelian group $M$, every quadratic form $q \in \operatorname{Quad}(G,M)$ can be expressed as $q(x) = c(x,x)$ for some bilinear form $c: G \times G \to M$.
\item For every abelian group $M$, the forgetful map $s: H^3_{2\text{-ab}}(G,M) \to H^3(G,M)$ is the zero map.
\item $G$ is 2-torsion free.
\end{enumerate}
\end{corollary}

\begin{proof}
We establish the equivalences by showing $(1) \Leftrightarrow (2) \Rightarrow (3) \Rightarrow (1)$.

The equivalence of (1) and (2) follows immediately from Lemma \ref{lem:equivalence_bilinear_representability}.

For $(3) \Rightarrow (1)$, we apply Theorem \ref{thm:individual_quadratic_criterion} directly.

For $(2) \Rightarrow (3)$, we use the characterization given in Proposition \ref{prop:image_suspension_map}: the forgetful map vanishes on $\mathbb{Q}/\mathbb{Z}$ if and only if $T_2(G)/2T_2(G) = 0$. Since this quotient is trivial precisely when $T_2(G) = 0$, condition (2) forces (3) to hold.
\end{proof}

\begin{remark}
Corollary \ref{cor:global_characterization} provides an affirmative answer to the question posed in \cite[Section 12]{Braunling21} about whether torsion-free abelian groups satisfy condition (1). Moreover, Theorem \ref{thm:individual_quadratic_criterion} gives a precise criterion for individual quadratic forms, and Proposition \ref{prop:image_suspension_map} characterizes the obstruction group explicitly as $\operatorname{Hom}(T_2(G)/2T_2(G), M)$.
\end{remark}

\subsection{Connection to homotopy 2-types and categorical groups}

We now establish connections between our results and the theory of categorical groups, which provide algebraic models for connected homotopy 2-types. For a connected homotopy 2-type $X$, the second Postnikov invariant $k_2 \in H^3_{2\text{-ab}}(\pi_1(X), \pi_2(X))$ corresponds to a quadratic form $q: \pi_1(X) \to \pi_2(X)$. The triviality of the first Postnikov invariant $k_1 \in H^3(\pi_1(X), \pi_2(X))$ is equivalent to the quadratic form being representable as the trace of a bilinear form.

\begin{corollary}\label{cor:individual_homotopy_criterion}
Let $X$ be a connected homotopy 2-type. The Postnikov invariant $k_1 \in H^3(\pi_1(X),\pi_2(X))$ is trivial if and only if the corresponding quadratic form $q: \pi_1(X) \to \pi_2(X)$ satisfies $|x|q(x) = 0$ for every element $x \in T_2(\pi_1(X))$, where $T_2(\pi_1(X))$ is the 2-primary component.
\end{corollary}

\begin{proof}
This follows immediately from Theorem \ref{thm:individual_quadratic_criterion}.
\end{proof}

\begin{corollary}\label{cor:homotopy_2_types}
Let $G$ be an abelian group. For any connected homotopy 2-type $X$ with $\pi_1(X) \cong G$, the first Postnikov invariant $k_1 \in H^3(\pi_1(X), \pi_2(X))$ is trivial if and only if $G$ is 2-torsion free.
\end{corollary}

\begin{proof}
By the correspondence between Postnikov invariants and quadratic forms, the condition that the Postnikov invariant is trivial is equivalent to the condition that the quadratic form $q: G \to \pi_2(X)$ satisfies $|x|q(x) = 0$ for all $x \in T_2(G)$. By Theorem \ref{thm:individual_quadratic_criterion}, this occurs for all possible $\pi_2(X)$ precisely when $G$ is 2-torsion free.
\end{proof}

\begin{remark}
The results of this subsection show that the 2-torsion of the fundamental group is the primary obstruction to triviality of Postnikov invariants for homotopy 2-types. This provides a group-theoretic criterion for when homotopy 2-types admit particularly simple algebraic models.
\end{remark}

\subsection{Explicit formulas for quadratic forms and 2-abelian 3-cocycles}

Let $G$ be a finitely generated abelian group with primary decomposition $G \cong T \oplus F$ where 
\[T = \bigoplus_{i=1}^k \mathbb{Z}/n_i\mathbb{Z} \quad \text{and} \quad F = \mathbb{Z}^r\]
with $n_1 | n_2 | \cdots | n_k$.

\begin{proposition}\label{prop:quad_structure}
The group of quadratic forms admits the decomposition
\begin{align*}
\operatorname{Quad}(G,M) \cong 
&\overbrace{
    \bigoplus_{i=1}^k M[\gcd(2n_i, n_i^2)] \oplus \bigoplus_{1 \le i < j \le k} M[n_i]
}^{\text{Quadratic forms on } T} \\
&\oplus \overbrace{
    M^r \oplus \bigoplus_{1 \le i < j \le r} M
}^{\text{Quadratic forms on } F} \\
&\oplus \overbrace{
    \bigoplus_{i=1}^k \bigoplus_{j=1}^r M[n_i]
}^{\text{Bilinear forms } T \times F \to M}.
\end{align*}
Moreover, any quadratic form $q: G \to M$ has the unique expression
\begin{align*}\label{eq:quadratic_formula}
q(a_1,\ldots,a_k,x_1,\ldots,x_r) = 
&\overbrace{\sum_{i=1}^k a_i^2 m_i}^{\text{quadratic on torsion}} 
+ \overbrace{\sum_{1 \le i < j \le k} a_i a_j b_{ij}}^{\text{torsion cross-terms}} \\
&+ \overbrace{\sum_{j=1}^r x_j^2 \ell_j}^{\text{quadratic on free}} 
+ \overbrace{\sum_{1 \le j < l \le r} x_j x_l f_{jl}}^{\text{free cross-terms}} 
+ \overbrace{\sum_{i=1}^k \sum_{j=1}^r a_i x_j t_{ij}}^{\text{torsion-free interaction}},
\end{align*}
where the parameters satisfy:
\begin{itemize}
\item $m_i \in M[\gcd(2n_i, n_i^2)]$ for each cyclic factor,
\item $b_{ij} \in M[n_i]$ for $1 \leq i < j \leq k$,
\item $\ell_j \in M$ for free quadratic terms,
\item $f_{jl} \in M$ for free cross-terms,
\item $t_{ij} \in M[n_i]$ for torsion-free interactions.
\end{itemize}
\end{proposition}

\begin{proof}
It follows by repeated application of Lemma \ref{lem:direct_sum_decomposition}, using that $\operatorname{Quad}(\mathbb{Z}/n\mathbb{Z},M)\cong M[\gcd(2n,n^2)]$ and $\operatorname{Bil}(\mathbb{Z}/n\mathbb{Z} \times \mathbb{Z}/m\mathbb{Z}, M) \cong M[\gcd(n,m)]$.
\end{proof}

Applying Lemma \ref{lem:direct_sum_construction} to a given quadratic form $q \in \operatorname{Quad}(G,M)$ with parameters as in Proposition \ref{prop:quad_structure}, we obtain the corresponding 2-abelian 3-cocycle $(h,c) \in Z^3_{2\text{-ab}}(G,M)$. Following \cite{Braunling21}, a 2-abelian 3-cocycle $(h,c)$ is in \emph{normal form} if $h(x,y,z) = c(x,y) + c(x,z) - c(x,y+z)$ and $h(z,x,y) = c(x+y,z) - c(x,z) - c(y,z)$ for all $x,y,z \in G$.

The work \cite{Braunling21} establishes the existence of normal form representatives for divisible coefficient groups $M$. Our approach provides explicit normal form representatives for arbitrary coefficient groups $M$, extending these results beyond the divisible case. This recovers the explicit formulas from \cite{Braunling21, Quinn, Torrecillas11braided, huang2020explicit, huang2014braided} and provides the following explicit construction.

Let $g = (a_1,\ldots,a_k,x_1,\ldots,x_r)$ and $g' = (a'_1,\ldots,a'_k,x'_1,\ldots,x'_r)$ be elements in $G$. The 2-cochain $c$ is given by

\begin{align*}
c(g,g') = &\overbrace{\sum_{i=1}^k a_i a'_i m_i + \sum_{1 \leq i < j \leq k} a_i a'_j b_{ij}}^{\text{Terms for } T} \\
&+ \overbrace{\sum_{j=1}^r x_j x'_j \ell_j + \sum_{1 \leq j < l \leq r} x_j x'_l f_{jl}}^{\text{Terms for } F} \\
&+ \overbrace{\sum_{i=1}^k \sum_{j=1}^r a_i x'_j t_{ij}}^{\text{Terms for } T \times F}.
\end{align*}

The 3-cocycle $h$ is given by the sum of the base 3-cocycles for the cyclic part of the torsion component:
\begin{equation*}
h(g,g',g'') = \sum_{i=1}^k h_{n_im_i}(a_i,a'_i,a''_i),
\end{equation*}
where $h_{n,m}: (\mathbb{Z}/n\mathbb{Z})^3 \to M$ is the base 3-cocycle
\begin{equation*}
h_{nm}(a,b,c) = \begin{cases}
anm, & \text{if } b + c \geq n, \\
0, & \text{otherwise}.
\end{cases}
\end{equation*}
The parameters correspond to those in Proposition \ref{prop:quad_structure}. 

\subsection{Symmetric quadratic forms and 3-abelian cohomology}

We conclude with the special case of quadratic forms corresponding to symmetric categorical groups. The group $H^3_{3\text{-ab}}(G,M)$ classifies categorical groups with symmetric structure. A 3-abelian 3-cocycle is precisely a 2-abelian 3-cocycle $(h,c)$ with the additional symmetry condition $c(x,y) + c(y,x) = 0$.

This symmetry condition implies $2c(x,x) = 0$ for all $x \in G$. Since $c(x,y) + c(y,x) = b_{\operatorname{Tr}(h,c)}(x,y)$ (see \ref{eq: S matrix}), we conclude that $q \in \operatorname{Quad}(G,M)$ corresponds to a 3-abelian 3-cocycle if and only if $q$ is a group homomorphism, or equivalently, $b_q = 0$. Therefore, the image of the natural inclusion $H^3_{3\text{-ab}}(G,M) \to H^3_{2\text{-ab}}(G,M)$ is precisely $\operatorname{Hom}(G, M[2]) \subset \operatorname{Quad}(G,M)$, the group of symmetric quadratic forms.

\begin{corollary}\label{cor:symmetric_quadratic_forms}
Every symmetric quadratic form $q \in \operatorname{Hom}(G, M[2])$ can be represented as the trace of a bilinear form.
\end{corollary}

\begin{proof}
If $x \in T_2(G)$, then $q(2x) = 2q(x) = 0$ since $q$ is a homomorphism and $q(x) \in M[2]$. In particular, $|x|q(x) = 0$ for all $x \in T_2(G)$.
\end{proof}

This result shows that symmetric quadratic forms have representatives with trivial 3-cocycle component, consistent with \cite[Theorem 2.2]{johnson2012modeling}: any symmetric categorical group admits a skeletal and strict realization.

For a finitely generated abelian group $G \cong \left(\bigoplus_{i=1}^k \mathbb{Z}/n_i\mathbb{Z}\right) \oplus \mathbb{Z}^r$, the group of symmetric quadratic forms decomposes as
\begin{equation*}
\operatorname{Hom}(G, M[2]) \cong \left(\bigoplus_{i: n_i \text{ even}} M[2]\right) \oplus \left(\bigoplus_{j=1}^r M[2]\right),
\end{equation*}
since $\operatorname{Hom}(\mathbb{Z}/n_i\mathbb{Z}, M[2]) = 0$ when $n_i$ is odd. Any symmetric quadratic form $q$ has the unique expression 
\begin{equation*}
q(a_1, \ldots, a_k, x_1, \ldots, x_r) = \sum_{i: n_i \text{ even}} a_i m_i + \sum_{j=1}^r x_j \ell_j
\end{equation*} 
with $m_i, \ell_j \in M[2]$. The corresponding 3-abelian 3-cocycle is $(0, c)$ where 
\begin{equation*}
c((a_1,\ldots,x_r), (a'_1,\ldots,x'_r)) = \sum_{i=1}^k a_i a'_i m_i + \sum_{j=1}^r x_j x'_j \ell_j.
\end{equation*}

Alternatively, symmetric quadratic forms arise as the quotient in the exact sequence
\begin{equation*}
0 \to \operatorname{Hom}(\Lambda^2 G, M) \to \operatorname{Bil}_{\text{skew}}(G \times G, M) \xrightarrow{\operatorname{Tr}} \operatorname{Hom}(G, M[2]) \to 0,
\end{equation*}
where $\operatorname{Bil}_{\text{skew}}(G \times G, M)$ denotes the group of skew-symmetric bilinear forms.
\subsection*{Acknowledgments}
The author wishes to thank BIMSA (Beijing Institute of Mathematical Sciences and Applications) for their hospitality during a visit where this work was initiated. The author is grateful to Sébastien Palcoux and Yi Long Wang for stimulating discussions that motivated this research, and to Christoph Schweigert and Oliver Braunling for valuable comments on an earlier version of this paper.


\end{document}